\documentclass[11pt]{amsart}
\usepackage{amsmath}
\usepackage{amssymb}
\usepackage{tikz-cd}
\usepackage{blkarray}
\usepackage{multirow}
\usepackage{hyperref}
\allowdisplaybreaks

\newcommand{\ot}{\otimes}
\newcommand{\op}{\oplus}

\newcommand{\D}{\operatorname{D}}
\newcommand{\Gr}{\operatorname{Gr}}
\newcommand{\GL}{\operatorname{GL}}

\newcommand{\R}{\mathbb{R}}
\newcommand{\id}{\operatorname{Id}}
\newcommand{\Hom}{\operatorname{Hom}}
\newcommand{\g}{\mathfrak{g}}
\newcommand{\gso}{\mathfrak{so}}
\newcommand{\cso}{\mathfrak{cso}}
\newcommand{\gsl}{\mathfrak{sl}}
\newcommand{\is}{\mathfrak{is}}
\newcommand{\C}{\mathbb{C}}
\newcommand{\I}{\operatorname{i}}
\newcommand{\e}{\operatorname{e}}
\newcommand{\tr}{\operatorname{tr}}
\newcommand{\diag}{\operatorname{diag}}
\renewcommand{\Im}{\operatorname{Im}}
\renewcommand{\Re}{\operatorname{Re}}
\newcommand{\Ad}{\operatorname{Ad}}
\newcommand{\gl}{\mathfrak{gl}}
\newcommand{\h}{\mathfrak{h}}

\newcommand{\Aff}{\operatorname{Aff}(\R^{n+1})}
\newcommand{\Af}{\operatorname{Aff}}

\newtheorem{thm}{Theorem}
\newtheorem*{thm*}{Theorem}
\newtheorem{prop}{Proposition}
\newtheorem{lem}{Lemma}
\newtheorem{cor}{Corollary}

\theoremstyle{definition}
\newtheorem{dfn}{Definition}

\newtheorem{rem}{Remark}

\title[Affine tube domains with large automorphism groups]{Classification of homogeneous affine tube domains with large automorphism groups in arbitrary dimensions}

\author{Vladimir Ezhov}
\address{V.E. Flinders University, College of Science and Engineering, 1284 South Road, Tonsley  SA 5042, Australia; 
MSU, Faculty of Mechanics and Mathematics, Leninskiye Gory 1,  Moscow GSP-1 119991,  Russia}
\email{vladimir.ejov@flinders.edu.au}
\author{Alexandr Medvedev}
\address{A.M. International School for Advanced Studies, via Bonomea 
265, Trieste 34136, Italy}
\email{amedvedev@sissa.it}

\author{Gerd Schmalz}
\address{G.S. University of New England, School of Science and Technology, Armidale NSW 2351, Australia}
\email{schmalz@une.edu.au}

\begin{document}

\begin{abstract}
We classify tube domains in $\mathbb C^{n+1}$ ($n\ge 1$) with affinely homogeneous base of their boundary and a.) with positive definite Levi form and b.) with Lorentzian type Levi form and affine 
isotropy of dimension at least $\frac{(n-2)(n-3)}2$.
\end{abstract}

\maketitle

\section{Introduction}
The study of homogeneous complex domains and their boundaries goes back to Elie Cartan \cite{Ca1,Ca2,Ca3}, who obtained a complete classification in $\mathbb C^2$. It was shown by Vinberg, Gindikin and Pyatetskii-Shapiro \cite{VGP} that any bounded homogeneous domain in $\C^n$ is holomorphically equivalent to a Siegel domain of the second kind, which can be viewed as a generalization of the upper half-plane to several dimensions. Therefore the remaining interesting case concerns homogeneous domains that are not equivalent to bounded domains.

Winkelmann classified all three-dimensional homogeneous complex manifolds in \cite{W}. In particular, he discovered a domain that is bounded by the Levi-indefinite hypersurface
$$\Im (w +z_1\bar{z}_2)=|z_1|^4,$$
where $(z_1,z_2,w)$ are coordinates in $\mathbb C^3$. This hypersurface features the largest possible symmetry algebra among the non-quadratic hypersurfaces.

Loboda classified all Levi non-degenerate hypersurfaces with 7-dimensional symmetry algebra \cite{L1,L2}, as well as all hypersurfaces with 6-dimensional symmetry algebra and positive definite Levi form \cite{L3}.

In a recent paper Doubrov, Medvedev and The \cite{DMT} prove that quadratic hypersurfaces and the Winkelmann surface are the only homogeneous hypersurfaces in $\mathbb C^3$ whose symmetry groups have open orbits in $\mathbb C^3$.

A complete classification of homogeneous domains in higher dimension becomes an unrealistic endeavour. Therefore the focus of research lies on special classes of domains. Penney \cite{P} constructed a class of homogeneous domains that generalize  homogeneous Siegel domains. They are called Siegel domains of type N-P or nil-balls.

Another approach to produce series of examples of unbounded homogeneous domains is to consider affinely homogeneous tube domains of the form $D= \Omega + \I \R^n$, where $\Omega$ itself is an affinely homogeneous domain in $\R^n$, which serves as affinely homogeneous base of the domain $D$. In this situation all affine automorphisms of $\Omega$ lift to automorphisms of $D$ together with the $n$-dimensional group of imaginary translations, and hence forms a transitive subgroup of affine transformations of $D$. Thus, the task of constructing $D$  reduces to constructing $\Omega$. In order to do so, we start with an affinely homogeneous hypersurface $\Gamma \subset \R^n$ considered as a part of a boundary of $\Omega$. In fact, $\Omega$ can be taken as an open orbit of the group $\Af(\Gamma)$ of affine automorphisms of $\Gamma$ lying on either side of $\Gamma$. 
This means that the dimension of $\Af(\Gamma)$ must be at least $n$ and hence there is a non-trivial affine isotropy of $\Gamma$. For $n=2$, Loboda~\cite{L} proved that any holomorphically homogeneous non-spherical tube hypersurface in $\C^2$ has an affinely homogeneous base. Up to dimension $n=4$ affinely homogeneous hypersurfaces in $\R^n$ with non-trivial isotropy have been classified (see \cite{NS}, \cite{DKR}, \cite{EE1}, \cite{EE2}, \cite{EE3}). Based on this classification Eastwood, Ezhov and Isaev constructed new examples of affinely homogeneous domains in  $\C^4$ (\cite{EEI}, \cite{EI}).

We use Cartan's moving frame method to classify affinely homogeneous surfaces with large symmetry algebra. The key observation, which considerably reduces the complexity, is Proposition \ref{P:tubeDom}. It provides a criterion, when a particular affine surface is a part of a boundary of an affinely homogeneous domain. For example, in $\C^4$ it reduces the analysis from $20$ types of homogeneous affine surfaces with isotropy (cf.~\cite{EE2}) to just $5$ types. All of them are indeed boundaries of affinely homogeneous domains. Using  Proposition \ref{P:tubeDom} we prove the Theorem below, which settles the pseudo-convex case.
\begin{thm}\label{thm1}
Consider a tube domain in $\C^{n+1}$ with an affinely homogeneous 
base. If the group of affine symmetries 
acts locally transitive on the part of the boundary with definite second fundamental form then this part of the boundary is affinely equivalent to 
\[ 
		x_{n+1}=\sum_{i=1}^n x_i^2.
\]
\end{thm}

The next interesting case to consider are surfaces with Lorentzian signature of the second fundamental form. We obtain a classification of such surfaces under the assumption that the dimension of the isotropy is large enough.

\begin{thm}\label{thm2}
Assume that an affinely homogeneous tube domain in $\C^{n+1},$ $n\ge4$ has an affinely homogeneous part of the boundary with a Lorentzian second fundamental form. If the dimension of its affine isotropy is at least $\frac{(n-2)(n-3)}2$ then the boundary is locally affinely equivalent to one of the listed below:
\begin{enumerate}
\item
$
x_{n+1} = x_1 x_n + \sum_{i=2}^{n-1} x_i^2;
$ 
\item
$
x_{n+1}=x_1 x_{n} + \sum_{i=2}^{n-1} x_i^2 + x_{1}^3;
$
\item\label{wink}
$
x_{n+1} = x_1 x_{n} + \sum_{i=2}^{n-1} x_i^2 + x_{1}^2 x_2 + 
\alpha 
x_{1}^4;
$
\item\label{typeC1}
$ x_{n+1}=x_1 x_{n} + x_1\sum_{i=2}^{n-1}
x_i^2;
$
\item
$0=x_{n+1} -2x_{n+1}x_n - 2 x_1 x_n^2  + x_1 x_{n} + 
\frac12 \sum_{i=2}^{n-1} x_i^2,$
 equivalently\newline
$0=(1-2x_n)(x_{n+1} + x_1 x_{n} +\frac12\sum_{i=2}^{n-1} x_i^2)  + 
x_n\sum_{i=2}^{n-1} x_i^2
 ;$
\item $ x_{n+1} = x_1 x_{n} + \sum_{i=2}^{n-1} 
x_i^2  + x_1 x_2^2;
$
\item \label{typeD2}
$0=
(1-2x_n)(x_{n+1} + x_1 x_{n} +\frac12\sum_{i=2}^{n-1} x_i^2)  + x_n
x_2^2. $
\end{enumerate}
\end{thm}
The hypersurfaces of type \eqref{wink} are higher-dimensional analogues the Winkelmann hypersurface and the hypersurfaces studied in \cite{EI}. For $\alpha=\frac{1}{12}$ these hypersurfaces are holomorphically equivalent to a quadric. The hypersurfaces with $\alpha>\frac{1}{12}$ and, analogously, $\alpha<\frac{1}{12}$, are holomorphically equivalent to each other. The homogeneous domains they bound are nilballs. 

Hypersurfaces \eqref{typeC1}--\eqref{typeD2} are new and generalise the hypersurfaces 
\begin{align*}
x_4 = x_1 x_3 + x_1x_2^2,
\\
x_4^2 = x_1 x_3 + x_1^2x_2,
\end{align*} found in \cite{EEI}.

The structure of the paper is as follows.
In the following section we introduce the notations for Cartan's moving frame method applied to homogeneous submanifolds. In Section \ref{S:Aff} we use Cartan's moving frame method to derive several lower order invariants for affine hypersurfaces. In Section \ref{S:bound} we introduce the notion of tubular affine hypersurface, namely a surface which is a part of a boundary of an affinely homogeneous domain and formulate a few results regarding tubular hypersurfaces. In Section~\ref{S:Lorenz} we proceed with the case of Lorentzian second fundamental form and prove Theorem~\ref{thm2}. In the last section we study hypersurface~\eqref{typeC1} from Theorem~\ref{thm2} in detail.  We compute its group of biholomorphic automorphisms and prove that it is not a boundary of a nil-ball.

\section{Cartan's Moving Frame Method}

In this section we introduce an algebraic version of Cartan's equivalence method following \cite{doukom}.
Consider a homogeneous manifold $M=G/G_0.$ The manifold $M$ carries the structure of a principle $G_0$-bundle $\pi:G\to M.$ On the Lie group $G$ there exists a unique left invariant differential form $\omega:TG\to\g$ such that $\omega_e:T_e G\to \g=\id.$ The form $\omega$ is called the Maurer-Cartan form. If $G$ is a subgroup of a matrix group then the explicit formula for $\omega$ is:
\[ \omega=g^{-1}dg. \]

Consider a submanifold $N$ of $M$. 
\begin{dfn}
A moving frame on $N$  is a section $s\colon N \to G$ of the restriction of the principle bundle $\pi:G\to M,$ i.e. $\pi\circ s=\id$.
\end{dfn}
The set of all frames over $N$ is a $G_0$-principle bundle. An arbitrary moving frame on $N$ has the form $s \cdot h,$
where $h:N\to G_0$ and $s$ is some fixed frame. The aim of Cartan's method is to construct a canonical frame for $N$. This is done by normalizing the pullback $s^*\omega$. The image of $s^*\omega$ is an $n$-dimensional linear subspace of $\g$, i.e. $s^*\omega$ defines a map $N\to\Gr_n(\g)$.  bundle If we change the frame $s$ to  $\bar s=s \cdot h,$
where $h:N\to G_0$ then
\[ \bar s^*\omega=(\bar s)^{-1} d \bar s=h^{-1}s^{-1} (d s) h+h^{-1}d h=h^{-1}(s^*\omega)  h+h^{-1}d h. \]
The summand $h^{-1} d h$ takes values in the Lie algebra $\g_0$ of $G_0$. Therefore, the change of the lift acts linearly on $s^*\omega$ $\operatorname{mod}\g_0$ by the adjoint action. The first step in the normalization procedure is to choose a canonical representative $V_1\in\Gr_n(\g/\g_0)$ in the orbit of $s^*\omega$ $\operatorname{mod}\g_0$. Then we consider only frames such that $\Im s^*\omega \equiv V_1$ $\operatorname{mod}\g_0$. By doing so we reduce the freedom in the choice of the frame to the group:
\[ G_1=\{g\in H\mid \Ad_g(V_1)=V_1\}. \]

If we assume that $N$ is of constant type, i.e. $\Im s^*\omega$ $\operatorname{mod}\g_0$ belongs to one orbit for all points of $N$ then we can proceed further. Notice that all homogeneous submanifolds are of constant type. The next step is to consider changes of frames with respect to the group $G_1$. The change $\bar s=s \cdot h,$ $h:N\to G_1$ induces an action on $\Im s^*\omega$ $\operatorname{mod}\g_1$ where $\g_1$ is the Lie algebra of $G_1$. Let $\pi\colon\g/\g_1\to\g/\g_0$ be the canonical projection. We choose a representative $V_2\in\Gr_n(\g/\g_1)$ in the orbit of $s^*\omega$ $\operatorname{mod}\g_1$ such that $\pi\colon V_2\to V_1$ is a linear isomorphism. This gives us the second reduction of the frame to the group
\[ G_2=\{g\in G_1\mid\Ad_g(V_2)=V_2\}. \]
Eventually, the sequence $G_i$ will stabilize to the isotropy subgroup $G_\infty$ of the surface $N$ and the obtained lift will be a canonical (defined up to the action of $G_\infty$) moving frame. 

The method described above can be formalized as follows for the equivalence problem of homogeneous submanifolds (cf. \cite{doukom}).   Let $H\subset G$ be the symmetry group of an affinely homogeneous 
submanifold $N$ and $\h$ be the corresponding symmetry Lie 
algebra. 
For an arbitrary point $p$ of the submanifold $N$ let $G_0$ be the stabilizer 
of $p$ and 
\[\g_0=\{X\in \g \mid X_p=0 \}\]
 be the corresponding Lie algebra. Consider the series of subalgebras
\begin{equation}\label{e:g}
 \g_{i+1}= \{ X\in \g_i \mid [X,\h]\in \g_i+\h \}.
\end{equation}
The Lie algebra $\g_i$ is in fact an infinitesimal stabilizer of the $i$\textsuperscript{th} jet of $N$ at the point $p$. It was shown in \cite{doukom} that in this setting we have
\begin{equation}\label{e:V} V_i=(\h + \g_i)/\g_i, i\ge 0 ,\quad V_\infty=(\h + 
\g_\infty)/\g_\infty. 
\end{equation}

An important observation is that in the homogeneous setting there is an additional restriction on $V_i$.
\begin{prop}\label{CS}
Consider a homogeneous submanifold with symmetry algebra $\h$. Let $\g_i$ and $V_i$ be defined as in \eqref{e:g} and \eqref{e:V}. For every $X+\g_{i+1},Y+\g_{i+1}\in V_{i+1}$ we have 
$[X,Y]+\g_{i}\in V_{i}$.
\end{prop}
\begin{proof}
Indeed, since 
\[ [\h+\g_{i+1},\h+\g_{i+1}]\subseteq 
\h+\g_{i+1}+[\h,\g_{i+1}] \subseteq \h+\g_{i}\] we see that $[X,Y]+\g_{i} 
\in (\h+\g_{i})/\g_{i} = V_{i}.$ 
\end{proof}
In order to classify homogeneous submanifolds of dimension $n$ in $M$ we use the following algorithm:
\begin{enumerate}
	\item Classify elements $V_0\in\Gr_n(\g/\g_0)$ up to the action 
	of $G_0$.
	\item For any $V_i\in\Gr_n(\g/\g_i)$ classified on the previous step let 
	$G_{i+1}\subseteq G_i$ be the stabilizer subgroup of $V_i$ 
	and $\g_{i+1}$ be the corresponding Lie algebra. If $\g_{i+1}\neq
	\g_i$ then classify all $V_{i+1}\in\Gr_n(\g/\g_{i+1})$ up to the action of $G_{i+1}$ such that $\pi( V_{i+1} )=V_{i}$ and  $V_{i+1}$ satisfies Proposition~\ref{CS}.
	\item
 If $\g_{i+1}=\g_i$ then $\g_\infty=\g_i$, and   $V_\infty=V_i$. The obtained symmetry algebra of the submanifold is $\h=\g_\infty\op V_\infty$ and the isotropy group is $G_\infty$.
\end{enumerate}
 
\section{Affinely homogeneous surfaces}\label{S:Aff}

Let $\Gamma$ be a homogeneous affine surface in $\R^{n+1}$. Since the surface is homogeneous and the Maurer-Cartan form $\omega$ is left-invariant it is enough to normalize a frame $s:\Gamma\to G=\Aff$ in one point. Let us fix the point $p=(0,\dots,0)^t$. Then the pair  $(G,G_0)$ is
\[ 
G = \begin{pmatrix}
A & B
\\
0 & 1
\end{pmatrix},\quad
G_0 =  \begin{pmatrix}
A & 0
\\
0 & 1
\end{pmatrix},
\quad A\in\GL_{n+1}(\R),B\in\R^{n+1}.
\]
The corresponding pair of Lie algebras is
\[ 
\g = \begin{pmatrix}
a & b
\\
0 & 0
\end{pmatrix},\quad
\g_0 =  \begin{pmatrix}
a & 0
\\
0 & 0
\end{pmatrix},
\quad a\in\gl_{n+1}(\R),b\in\R^{n+1}.
\]

We recall that the normalization of the frame $s$ consists in a step-by-step choice of elements $V_i\in\Gr(\g/\g_i)$ in the orbit of $s^*\omega$ $\operatorname{mod}\g_i$ under the action of $G_i$. At the first step, when $i=0$, we consider $\Gr_n(\g/\g_0)=\Gr_n(\R^{n+1})$ under the action of $G_0=\GL_{n+1}(\R)$. Since $\GL_{n+1}(\R)$ acts transitively on hyperplanes in $\R^{n+1}$ we can choose
\[ V_0 = \left(
\begin{array}{cc}
  * & X
 \\
* & 0 \\
 0 & 0 
\end{array}
\right), \quad X\in \R^{n},\] 
where $X=(x_1,x_2,\dots,x_n)^t$ is an arbitrary vector in $\R^n$. 
The stabilizer Lie group for $V_0$ and the corresponding Lie algebra  are
\begin{align*}
G_1 &= \begin{pmatrix}
A_1 & D & 0
\\
0 & C & 0
\\
0 & 0 & 1
\end{pmatrix},\quad A_1\in\GL_n(\R), D\in\R^n, C\in\R^*
\\
\g_1 &=  \begin{pmatrix}
a_1 & e & 0
\\
0 & c & 0
\\
 0 & 0 & 0
\end{pmatrix},
\quad a_1\in\gl_n(\R),e\in\R^n, c\in\R
\end{align*}

Let us proceed with the second step of the moving frame normalization. All $n$-dimensional planes $V_1$ in $\g/\g_1$ that satisfy condition $\pi(V_1)=V_0$ have the form
\begin{equation}\label{e:P}
\begin{pmatrix}
* & * & X
\\
(PX)^t & * & 0
\\
 0 & 0 & 0
\end{pmatrix}, \quad P\in\operatorname{Mat}_{n\times n},
\end{equation} 
i.e. they are determined by some matrix $P$.
\begin{lem}
For a homogeneous surface the matrix $P$ is symmetric.	
\end{lem}
\begin{proof}
Indeed, consider two vectors  from $V_1$: 
\[ v_1=\begin{blockarray}{(ccc)}
* & * & X  \\
(PX)^t & * & 0 \\
0 & 0 & 0
\end{blockarray}, \quad
v_2 = \begin{blockarray}{(ccc)}
* & * & Y  \\
(PY)^t & * & 0 \\
0 & 0 & 0
\end{blockarray}. \]
Their commutator is
\[ [v_1,v_2]
= 
\begin{blockarray}{(ccc)}
* & * & *  \\
* & * & X^t P^t Y - Y^t P^t X \\
0 & 0 & 0
\end{blockarray}.
 \]
Proposition \ref{CS} requires $[v_1,v_2]\in V_0+\g_0$. Therefore,  $ X^t P^t Y - 
 Y^t P^t X = 0$, i.e., the matrix $P$ is symmetric.
 \end{proof}
 
 The symmetric form $P$ corresponds to the second fundamental form of the surface $\Gamma$. In what follows we restrict ourself to surfaces with a non-degenerate second fundamental form, i.e. with non-degenerate $P$.
 The adjoint action of $G_1$ on \eqref{e:P} induces the following change of $P$:
 \[ P \to C^{-1} A_1^t P A_1 . \]
 Assume that the form $P$ has signature $(p,q)$ with  $p\ge q$. Then $P$ can be normalized to 
 \[I_{p,q}=\begin{pmatrix}
 0 & 0 & I_q \\
 0 & I_{p-q} & 0 \\
 I_q & 0 & 0
 \end{pmatrix}.\] 
The resulting second order normalization of the frame is
\[ V_1 = \left(
\begin{array}{ccc}
  * & * & X
 \\
(I_{p,q}X)^t & * & 0 \\
 0 & 0 & 0
\end{array}
\right).\] 
The stabilizer Lie group for $V_1$ and the corresponding Lie algebra  are
\begin{align*}
G_2 &= \begin{pmatrix}
E\cdot A_2 & D & 0
\\
0 & E^2 & 0
\\
0 & 0 & 1
\end{pmatrix},\quad A_2\in O_{p,q}(\R), D\in\R^n, E\in \R^*_+
\\
\g_2 &=  \begin{pmatrix}
a_2 +e & e & 0
\\
0 & 2 e & 0
\\
 0 & 0 & 0
\end{pmatrix},
\quad a_2\in\gso_{p,q}(\R),e\in\R^n, e\in\R
\end{align*}

Let us proceed with the third step of the normalization procedure. All $V_2\in\Gr_n(\g/\g_2)$ that satisfy condition $\pi(V_2)=V_1$ can be represented as
 	\[ V_2 = \begin{pmatrix}
 	L_1(X) & * & X \\
 	(I_{p,q}X)^t & * & 0  \\
 	0 & 0 & 0  \\
\end{pmatrix} \]
where  $L_1(X)\in \gl_n(\R) / \gso_{p,q}(\R)$. We can assume that for every $X$ the matrix $L_1(X)$ is symmetric with respect to
$I_{p,q}$, i.e.:
\begin{equation}\label{eq:sym}
L_1(X)^t I_{p,q} =  I_{p,q} L_1(X).
\end{equation}
The symmetric with respect to $I_{p,q}$ part of an arbitrary matrix $T$ is given by the formula
\[ T_{sym}=\frac12(T+I_{p,q} T I_{p,q}).
\]
Under the action of $G_2$ the tensor $L_1$ changes as a $(2,1)$-tensor. Its structure for homogeneous surface is given by the following lemma.
\begin{lem}
For homogeneous surfaces $L_1\colon V_n\otimes V_n 
\to V_n$ is self-adjoint with respect to $I_{p,q}$ and 
symmetric in its arguments.	
\end{lem}
\begin{proof}
Consider any 2 vectors from $V_2$:
\[ 
v_1 = \begin{blockarray}{(ccc)}
L_1(X) & * & X  \\
( I_{p,q} X)^t & * & 0 \\
0 & 0 & 0
\end{blockarray}, \quad
v_2 =
\begin{blockarray}{(ccc)}
L_1(Y) & * & Y  \\
( I_{p,q} Y)^t & * & 0 \\
0 & 0 & 0
\end{blockarray}.
\]
Their commutator is
\begin{equation}\label{e:L1}
[v_1,v_2] = \begin{blockarray}{(ccc)}
* & * & L_1(X)Y- L_1(Y)X  \\
X^t I_{p,q}  L_1(Y) - Y^t  I_{p,q} L_1(X) & * & 0 \\
0 & 0 & 0
\end{blockarray}.
\end{equation}
If we use the notation $Z=L_1(X)Y- L_1(Y)X$ then Proposition~\ref{CS} implies 
that \eqref{e:L1} has the form
\[ \begin{pmatrix}
* & * & Z  \\
 (I_{p,q} Z)^t & * & 0 \\
0 & 0 & 0
\end{pmatrix},\]
which is equivalent to
\[ X^t I_{p,q}  L_1(Y) - Y^t  I_{p,q} L_1(X) 
=
\left(L_1(X)Y- L_1(Y)X\right)^t I_{p,q}.
\]
Using \eqref{eq:sym} we obtain that $2\left(L_1(X)Y-L_1(Y)X\right)^t 
I_{p,q}=0$, i.e. $L_1$ is symmetric in its arguments.
\end{proof}
Consider elements in $G_2$ of the form
\begin{equation}\label{e:D} \begin{pmatrix}
1_n & D & 0
\\
0 & 1 & 0
\\
0 & 0 & 1
\end{pmatrix},\quad D\in\R^n. 
\end{equation}
The action of \eqref{e:D}
on $\tr L_1$ is given by 
\[\tr L_1(X) \to \tr L_1(X)+\frac{n-2}2 
(D,X)_{p,q}.\] We 
can normalize $L_1$ to a trace-free tensor since the bilinear form $I_{p,q}$ is non-degenerate. The equivalence class of the trace-free part of $L_1$ 
under the action of $CO_{p,q}^+=O_{p,q} \times \R^*_+$ is the third order invariant 
for homogeneous affine surfaces. In particular, there is a unique homogeneous affine surface with the trace-free part of $L_1$ equal to $0$.
\begin{prop}\label{P:cones}
A homogeneous surface with the second fundamental form of signature $(p,q)$ is 	isomorphic to 
	\begin{equation} \label{srf:1}
	x_{n+1}=\sum_{i=1}^p x_i^2 - \sum_{i=p+1}^n x_i^2 
	\end{equation}
if and only if the trace-free part of $L_1$ is $0$. The surface \eqref{srf:1} has the isotropy group $CO_{p,q}^+=O_{p,q} \times \R^*_+$ and the isotropy groups of all Levi non-degenerate homogeneous surfaces of signature $(p,q)$ are contained in $CO_{p,q}^+$.
\end{prop}
\begin{proof}
If the trace-free part of $L_1$ is $0$ then the Cartan method terminates and the
transitive part of the symmetry Lie 
algebra is equivalent to 
\[\begin{pmatrix}
0 & 0 & X \\
X^t I_{p,q} & 0 & 0 \\
0 & 0 & 0
\end{pmatrix}. \]
The exponent of this matrix is
\begin{equation}
\label{eq2}
\begin{pmatrix}
1_n & 0 & X \\
X^t I_{p,q} & 1 & \frac12\sum_{i=1}^p x_i^2 - \frac12\sum_{i=p+1}^n 
x_i^2 \\
0 & 0 & 1
\end{pmatrix}. 
\end{equation}
The action of \eqref{eq2} gives the surface \eqref{srf:1} after 
rescaling 
$x_n$.
\end{proof}
Using $I_{p,q}$, the trace-free part of $L_1$ can be identified with a symmetric trace-free 3-tensor.  To complete the third step of Cartan's moving frame method we need to classify these tensors up to the action of $CO_{p,q}^+$. However let us first study some special properties of the surfaces that are parts of boundaries of homogeneous affine domains.

\section{Boundary of homogeneous affine domains}\label{S:bound}

Our final goal is the classification of affinely homogeneous tube domains with Levi non-degenerate boundary in $\C^{n+1}.$  Affine surfaces that are parts of boundaries of affinely homogeneous domains in $\R^n$ have $L_1$ of a special type, as the following proposition shows.

\begin{prop}\label{P:tubeDom}
A homogeneous affine surface with non-zero trace-free semi-invariant $L_1$ is a part  of a boundary of a homogeneous affine domain if and only if the projection of the isotropy subalgebra $\h_0\subset\cso_{p,q}(\R)$ on the center of $\cso_{p,q}(\R)$ is non-zero. Or, equivalently, if 
the projection of $\h_0$ on $\gso_{p,q}(\R)$ acts on $L_1$ by scalars that are not all zeros.
\end{prop}  	
\begin{proof}
At the point $p=(0,\dots,0)^t\in\Gamma$ the transitive part of the symmetry algebra $\h$ is spanned by  $V_\infty/\g_0=V_0/\g_0$. This implies that $T_p\Gamma$ is spanned by $\frac{\partial}{\partial x_i}$, $1\le i \le n$. We are going to check the action of the symmetry algebra $\h$ at the point 
$p_t=(0,\dots,0,t)^t$ that does not belong to $\Gamma$ for sufficiently small $t\neq0$. To obtain the infinitesimal action of $\h$ at the point $p_t$ we need to conjugate Lie algebra $\h$ by 
\[ A_t=\begin{pmatrix}
1_n & 0 & 0 \\
0 & 1 & t \\
0 & 0 & 1 
\end{pmatrix}.\]	
The transitive at $0$ part of $\h$ takes at $p_t$ the form
\[ A_t^{-1}.\begin{pmatrix}
L(X) & L_3(X) & X \\
X^t I_{p,q} & L_4(X) & 0 \\
0 & 0 & 0 
\end{pmatrix} . A_t =
\begin{pmatrix}
L(X) & L_3(X) & X+t L_3(X) \\
X^t I_{p,q} & L_4(X) & t L_4(X) \\
0 & 0 & 0 
\end{pmatrix},
\]
where $L(X)=L_1(X)+L_2(X)$, $L_2(X)\in\gso_{p,q}(\R).$ An arbitrary element of $\h_0$ transforms to 
\[ A_t^{-1}.\begin{pmatrix}
a_2 + e I_n & 0 & 0 \\
0 & 2 e & 0 \\
0 & 0 & 0
\end{pmatrix} . A_t =
\begin{pmatrix}
 a_2 + e I_n & 0 & 0 \\
0 & 2 e & t 2 e \\
0 & 0 & 0
\end{pmatrix}.
\]
To have a transitive action of $\h$ in a neighborhood of $p_t$ we need a projection of $\Ad_{A_t}(\h)$ on $\g/\g_0$ to be surjective. 
This is true only if $e\neq0$ for at least one element of 
$\h_0$.

Let us prove the second statement of the proposition. We remind that $G_3$ (and therefore $G_\infty$) consists of elements that fix $L_1$. Therefore an element of $\h_0=\g_\infty$ should annihilate $L_1$. Then, since  
\[ 
\begin{pmatrix}
 e I_n & 0 & 0 \\
0 & 2 e & 0 \\
0 & 0 & 0
\end{pmatrix}
\] 
scales $L_1$ by $-e$, the action of $a_2$ should scale $L_1$ by $e$.
\end{proof}
\begin{cor}\label{cor1}
	All homogeneous tubular affine surfaces have $\|L_1\|_{(p,q)}=0$, 
	where the pseudo-norm is induced from $V$ to $\Hom(S^2(V),V)$.
\end{cor}
\begin{proof}
Consider an element of the form
\[ Z = \begin{pmatrix}
a_2 + e I_n & 0 & 0 \\
0 & 2 e & 0 \\
0 & 0 & 0
\end{pmatrix}  \]
with $e\neq 0$ from the isotropy subalgebra of the surface. Such an element always exists according to Proposition \ref{P:tubeDom}. If $\|L_1\|\neq 0$ then $Z$ scales the norm of $L_1$ by $e$ and therefore cannot preserve the tensor $L_1$.
\end{proof}

Theorem 1 is a direct consequence of Corollary \ref{cor1}. Indeed, if the second fundamental form is definite then $\|L_1\|=0$ implies $L_1=0$. All surfaces with non-degenerate second fundamental form and $L_1=0$ are described in Proposition \ref{P:cones}.

\section{ Affine surfaces with Lorentzian second fundamental form} \label{S:Lorenz}

In this section we prove Theorem \ref{thm2}. 
From now on we assume that all homogeneous affine surfaces we study are parts of a boundary of an affinely homogeneous domain unless specified otherwise. We continue the 3rd 
step of Cartan's equivalence method assuming that the trace-free part of $L_1$ is non-zero. 
We normalize $L_1$ by scaling and choosing an 
appropriate representative in the $O_{p,q}(\R)$ orbit of $L_1$. According to Proposition \ref{P:tubeDom}, we need to find a tensor in $S^3(V)=S^3(\R^{p,q})$ on which at least one element from $\gso_{m-1,1}$ acts by a non-zero scalar.

We recall that the form $I_{n-1,1}$ is defined by
\[ (x_1,x_n)=1,\quad (x_i,x_i)=1,\mbox{ for }2\le i \le n-1 \]
and all other pseudo-scalar products between elements of the basis $\{ x_i \}$ are zero.
\begin{prop}\label{P:L1}
Consider a homogeneous affine surface $\Gamma$ with Lorentzian second fundamental form that is a part of a boundary of an affinely homogeneous domain. If the trace-free part of $L_1$ for $\Gamma$ is not zero then $L_1$ lies in the the orbit of exactly one element from the list
\begin{enumerate}
\item $ x_1^3 $,
\item $ 3 x_1^2 x_2 $,
\item $ 3 x_1\left(\sum_{i=2}^{n-1}\alpha_i x_i^2\right),$ where $\alpha_2=1$, $\alpha_i\ge \alpha_{i+1}$ and $|\alpha_i|\le 1$,
\end{enumerate}
under the action of $G_2=CO_{p,q}^+(\R)\times \R^n$.
\end{prop}
\begin{proof}
Consider $T\in\gso_{m-1,1}$ that acts as a non-zero scalar on $L_1$. 
Let $T=T_s+T_n$ be the abstract Jordan-Chevalley decomposition of $T$, where 
$T_s$ is semisimple, $T_n$ is nilpotent and $[T_s,T_n]=0$. This 
decomposition is unique. The semisimple element $T_s$ belongs to a maximal toral subalgebra. In $\gso_{n-1,1}$ there are just two of them: the maximally compact one and the maximally non-compact one. Elements of the maximally compact toral subalgebra can not have real eigenvalues other than $0$ in any representation of $\gso_{n-1,1}$. Consequently they can not scale $L_1$. Therefore and due to Proposition \ref{P:tubeDom} the semisimple part  $T_s$  belongs to the maximally non-compact toral sub-algebra. Since all maximally non-compact toral sub-algebras are conjugate we can assume that
\[ T_s=\begin{pmatrix}
\lambda & 0 & 0 \\
0 & T_{m-2} & 0 \\
0 & 0 & -\lambda
\end{pmatrix}
\]
where $T_{m-2}\in \gso_{m-2}(\R)$, $\lambda\neq 0$. All commuting with $T_s$ elements of $\gso_{m-1,1}$ are of the form
\[ \begin{pmatrix}
\tilde\lambda & 0 & 0 \\
0 & \tilde T_{m-2} & 0 \\
0 & 0 & -\tilde\lambda
\end{pmatrix}
\]
and therefore $T_n$ is of this form as well. But this means that $T_n$ is in fact semisimple, i.e, $T_n=0$. 

Let $T=T_s$ with $\lambda=1$.
Tensor $L_1$ should belong to one of the real eigenspaces of $T$. All 
eigenvalues of the action of $T$ on $S^3(V)$ are $\pm 3$, $\pm 2$, $\pm 
1$ and $0$. The corresponding eigenvectors are $x_1^i x_n^j f(\bar x)$ where $\bar x=(x_2,\dots,x_{n-1})^t$. 
The eigenvalue of $x_1^i x_n^j f(\bar x)$ is $(i-j)$. The group $O_{m-1,1}(\R)$ 
contains an element that exchanges $x_1$ and $x_n$. 
Since $T$ acts as a non-zero scalar the analysis reduces to orbits in the positive eigenvalue spaces.
We have 3 different cases:
\begin{enumerate}
\item \textbf{Eigenvalue $3$.}\label{c:1} The only eigenvector is 
proportional to $x_1^3$. The stabilizer group is a minimal parabolic subgroup of $O_{m-1,1}$,  i.e. a minimal subgroup of $O_{m-1,1}$ the complexification of which contains a Borel subgroup.
\item \textbf{Eigenvalue $2$.} Eigenvectors are of the form 
$x_1^2f(\bar x)$ where $f$ is linear. One can normalize any such vector to 
$x_1^2 x_2$ using the action of $O_{n-2}(\R)$. The stabilizer group is the subgroup of a minimal parabolic from \eqref{c:1} that fixes $x_1$.
\item \textbf{Eigenvalue $1$.} Eigenvectors are of the form 
$x_1q(\bar x)+\beta x_1^2 x_n$ where $q$ is a quadratic form. In this particular case it is 
convenient to normalize $L_1$ not to trace-free form but to the form 
with $\beta=0$ . We normalize $q$ is
to $q(\bar x)=3\sum_{i=2}^{n-1} \alpha_i x_i^2$ by 
orthogonal transformations and we scale the first coefficient $\alpha_1$ to~$1$. The stabilizer of $L_1$ is the subgroup of a minimal parabolic group from \eqref{c:1} that fixes $q$.
\end{enumerate}
\end{proof}
Proposition \ref{P:L1} completes the third step of Cartan's moving frame method. 
The stabilizer Lie groups for $V_2$ and the corresponding Lie algebras  are
\begin{align*}
G_3 &= \begin{pmatrix}
E\cdot A_3 & 0 & 0
\\
0 & E^2 & 0
\\
0 & 0 & 1
\end{pmatrix},\quad A_3\in O_{p,q}(\R), A_3.L_1=e\cdot L_1
\\
\g_3 &=  \begin{pmatrix}
a_3 +e & 0 & 0
\\
0 & 2 e & 0
\\
 0 & 0 & 0
\end{pmatrix},
\quad a_3\in\gso_{p,q}(\R), a_3.L_1=e\cdot L_1
\end{align*}

We proceed with the 4th step of Cartan's equivalence method 
using the obtained tensors $L_1$. We will need an action of the scaling 
component of $CO_{p,q}^+(\R)$ on the transitive part of the symmetry algebra.  Using 
the action of
\[ A_E=\begin{pmatrix}
  E I_n & 0 & 0 \\
0 & E^2 & 0\\
0 & 0 & 1
\end{pmatrix}
\]
we have
\[\Ad_{A_E^{-1}} \begin{pmatrix}
L(X) & L_3(X) & X \\
X^t I_{n-1,1} & L_4(X) & 0 \\
0 & 0 & 0 
\end{pmatrix}  = \begin{pmatrix}
L(X) & E L_3(X) & \frac1E X \\
\frac1E X^t I_{n-1,1} & L_4(X) & 0 \\
0 & 0 & 0 
\end{pmatrix},
\]
where $L(X)=L_1(X)+L_2(X)$, $L_2(X)\in\gso_{p,q}(\R).$ Substituting $\frac1E X$ with $Y$ we obtain the answer
\[ \begin{pmatrix}
E L(Y) & E^2 L_3(Y) & Y \\
Y^t I_{n-1,1} & E L_4(Y) & 0 \\
0 & 0 & 0 
\end{pmatrix}.
\]

Consider the three cases from Proposition \ref{P:L1}.

\textbf{1. $L_1=x_1^3$.} The stabilizer $G_3$ of $L_1$ coincides with the 
stabilizer of $x_1$, which is
\begin{equation}\label{s3:e1} \begin{pmatrix}
  Z^3 A_3 & 0 & 0 \\
0 & Z^6 & 0\\
0 & 0 & 1
\end{pmatrix}  
\end{equation}
where 
\[ A_3 = \begin{pmatrix}
  Z & 0 & 0  \\
0 & D_{n-2} \\
0 & 0 & Z^{-1}
\end{pmatrix} .
\begin{pmatrix}
  1 & -Q^t_{n-2} & -|Q_{n-2}|^2/2 \\
0 & 1 & Q_{n-2}\\
0 & 0 & 1
\end{pmatrix},
 \]
and  $D_{n-2}\in O_{n-2}(\R),$ $Q_{n-2}\in \R^{n-2},$ $Z\in R^*_+$. The 
corresponding 
Lie algebra has the form
\[ 
\g_3=\begin{pmatrix}
 4z & -q^t & 0 & 0 & 0 \\
0 & 3z + d_{n-2} & q & 0 & 0\\
0 & 0 & 2z & 0 & 0 \\
0&0&0& 6z &0 \\
0&0&0&0&0 
\end{pmatrix}.
 \]
We find that all possible $V_3\in\Gr_n(\g/\g_3)$ are
 	\[ V_3 =  \begin{pmatrix}
 	 	L_1(X)+L_2(X) & L_3(X) & X \\
 	 	(I_{n-1,1}X)^t & * & 0  \\
 	 	0 & 0 & 0  \\
 	\end{pmatrix},  \]
 	where $L_3\in \Hom(\R^n,\R^n)$ and $L_2(X)$ has the form
\[ \begin{blockarray}{cccc}
{\scriptstyle 1} & {\scriptstyle n-2} &  {\scriptstyle 1} \\
\begin{block}{(ccc)c}
W(X) & 0 & 0 & {\scriptstyle 1} \\
P(X) & 0 & 0 & {\scriptstyle n-2} \\
0 & -P(X)^t & -W(X) &  {\scriptstyle 1} \\
\end{block}
\end{blockarray}. \]
According to the proof of Proposition \ref{P:L1}, the isotropy group $G_\infty$ and therefore the stabilizer $G_4$ of $V_3$ 
contains an element $T$ of the form \eqref{s3:e1} with $Z\neq 1$ and
\[  A_3 = \begin{pmatrix}
  Z & 0 & 0  \\
0 & D_{n-2} \\
0 & 0 & Z^{-1}
\end{pmatrix} .\]  
In order to preserve $V_3$ the element $T$ must preserve the tensors $L_2$ and $L_3$. Therefore, since
\[ T_1=\begin{pmatrix}
  Z^3 & 0 & 0 \\
0 & Z^6 & 0\\
0 & 0 & 1
\end{pmatrix}  \]
scales $L_2$ by $Z^3$, $L_3$ by $Z^6$ and $T$ is a product of $T_1$ and $A_3$, the element $A_3$ scales $L_2$ 
and $L_3$ by $Z^{-3}$ and $Z^{-6}$, respectively. But this is impossible unless $L_2= 0$, $L_3=0$, since $L_2$, $L_3$ have eigenvalues with respect to $A_3$ in the range $\{K^2,K,1,K^{-1},K^{-2}\}$.
We conclude that $L_2$ and $L_3$ are $0$ and Cartan's equivalence method terminates in this case. As a result there exists only one tubular surface with $L_1=x_1^3$ which is a part of a boundary of a homogeneous domain.

The normal form of the equation with symmetry algebra $\h=V_\infty\oplus\g_\infty=V_2\oplus\g_3$ is
\begin{equation}\label{srf:2}
 x_{n+1}=x_1 x_{n} + \sum_{i=2}^{n-1} x_i^2 + x_{1}^3;
\end{equation}

\textbf{2. $L_1=3 x_1^2 x_2$.}  The stabilizer $G_3$ of $L_1$ in this 
case 
consists of transformations fixing $x_1$ and $x_2$:

\begin{equation}\label{s3:e2} \begin{pmatrix}
  Z^2 A_3 & 0 & 0 \\
0 & Z^4 & 0\\
0 & 0 & 1
\end{pmatrix},  
\end{equation}
where 
\[ A_3 = \begin{pmatrix}
  Z & 0 & 0 & 0  \\
  0 & 1 & 0 & 0 \\
0 & 0 & D_{n-3} &  0 \\
0 & 0 & 0 & Z^{-1}
\end{pmatrix} .
\begin{pmatrix}
1 & 0 & -Q^t_{n-3} & -|Q_{n-3}|^2/2 \\
0& 1 & 0 & 0 \\  
0 & 0 & 1  & Q_{n-3}\\
0 & 0 & 0 & 1
\end{pmatrix}
 \]
and  $D_{n-3}\in O_{n-3}(\R),$ $Q_{n-3}\in \R^{n-3},$ $Z\in R^*_+$. The 
corresponding 
Lie algebra has the form
\[ 
\g_3=\begin{pmatrix}
 3z & 0 & -q^t_{n-3} &  0 & 0 & 0 \\
 0 & 2z & 0 & 0 & 0 & 0\\
0 & 0 & 2z + d_{n-3} & q_{n-3} & 0 & 0\\
0 & 0 & 0 & z & 0 & 0 \\
0&0&0& 0 & 4z & 0 \\
0&0&0&0&0 & 0
\end{pmatrix}.
 \]
We find that all possible $V_3\in\Gr_n(\g/\g_3)$ are
 	\[ V_3 =  \begin{pmatrix}
 	 	L_1(X)+L_2(X) & L_3(X) & X \\
 	 	(I_{n-1,1}X)^t & * & 0  \\
 	 	0 & 0 & 0  \\
 	\end{pmatrix},  \]
 	where $L_3\in \Hom(\R^n,\R^n)$ and $L_2$ has the form
\[ \begin{blockarray}{ccccc}
{\scriptstyle 1} & {\scriptstyle 1} & {\scriptstyle n-3} &  
{\scriptstyle 1} \\
\begin{block}{(cccc)c}
Z(X) & -Q_1(X) & 0 & 0 & {\scriptstyle 1} \\
P_1(X) &  0 & -S(X)^t & Q_1(X) & {\scriptstyle 1} \\
P_2(X) &  S(X) & 0 & 0 &  {\scriptstyle n-3} \\
0 & -P_1(X) & -P_2(X)^t & -Z(X) &  {\scriptstyle 1} \\
\end{block}
\end{blockarray} .\]

By the same argument as in the first case, the isotropy group contains at least one element \eqref{s3:e2} with
\[ A_3 = \begin{pmatrix}
  Z & 0 & 0 & 0  \\
  0 & 1 & 0 & 0 \\
0 & 0 & D_{n-3} &  0 \\
0 & 0 & 0 & Z^{-1}
\end{pmatrix}. \]
Since
\[ \begin{pmatrix}
  Z^2 & 0 & 0 \\
0 & Z^4 & 0 \\
0 & 0 & 1
\end{pmatrix}  \]
scales $L_2$ by $Z^2$ and $L_3$ by $Z^4$, $A_3$ must scale $L_2$ and $L_3$ by $Z^{-2}$ and $Z^{-4}$. Checking the eigenvalues we see that $L_3=0$ and the only possible $L_2$ has $Q_1= \beta\cdot x_n^*\ot x_1$ and all other parts equal to $0$. Since $G_3$ already preserves 
$ x_n^*\ot x_1 \sim x_1^2$ the Cartan equivalence method stops. We obtain a 1-parameter family of tubular surfaces.

The normal form of the corresponding equation is 
\begin{equation}\label{srf:3} 
x_{n+1} = x_1 x_{n} + \sum_{i=2}^{n-1} x_i^2 + x_{1}^2 x_2 + 
\alpha 
x_{1}^4;
\end{equation}
with $\alpha=\frac17\beta.$

\textbf{3. $L_1=3x_1\left(\sum_{i=2}^{n-1}\alpha_i 
x_i^2\right)=3x_1q(\bar x)$.}

First we 
notice that $q(\bar x)$ does not
decompose into a product of two linear terms, one of which is light-like.
As a consequence the stabilizer $G_3$ must preserve $x_1$ as well as the quadratic form $q$.
Consider the action of \[ K = \begin{pmatrix}
  Z & 0 & 0  \\
0 & D & 0\\
0 & 0 & Z^{-1}
\end{pmatrix} .
\begin{pmatrix}
  1 & -Q^t & -|Q|^2/2 \\
0 & 1 & Q\\
0 & 0 & 1
\end{pmatrix},
 \] on $\alpha=\diag( \alpha_1,\dots,\alpha_{m-2})\in 
 \Hom(\R^{m-2},\R^{m-2})$ which represents the bilinear form $q$:
 \[ K^t  \begin{pmatrix}
 0 & 0 & 0 \\
 0 & \alpha & 0 \\
 0 &0 & 0
 \end{pmatrix} K =
 \begin{pmatrix}
 0 & 0 & 0 \\
  0 & D^t\alpha D & Z^{-1} D^t Q^t\alpha Q \\
  0 & Z^{-1}  Q \alpha Q^t D & Z^{-2} \sum_i \alpha_i(Q_i)^2)
  \end{pmatrix}.\]
   Since $Z^{-1} D^t  Q^t\alpha Q =0$ iff $ Q^t\alpha Q=0$ we obtain that 
   $Q_i\neq 0$ only if $\alpha_i=0$. This condition also implies 
   $\sum_i \alpha_i(Q_i)^2=0$. Therefore, the stabilizer $G_3$ of 
   $L_1$ consists of elements 
   \[\begin{pmatrix}
     Z A_3 & 0 & 0 \\
   0 & Z^2 & 0\\
   0 & 0 & 1
   \end{pmatrix}, \]
   where $A_3$ is formed by elements $K$ with the restrictions: $Q_i = 0$ if $\alpha_i\neq0$, 
   $D\in \prod_k O(\R^{i_k})$ and $i_k$ is the number of equal 
   $\alpha_i$. Note that in the generic case, when $\alpha_i\neq 0$, 
   $\alpha_i\neq\alpha_j$ for $i\neq j$, the stabilizer is 1-dimensional. There are 
   only two subgroups of dimension $\frac{(n-2)(n-3)}2+1$ for $n\ge 4$ that can stabilize some $\alpha$. First, the 
   stabilizer of $\alpha=\diag(1,0,\dots,0)$, which is 
   $CO_{n-3}^+(\R)\times\R^{n-3}$ and second, the 
   stabilizer of $\alpha=\diag(1,1,\dots,1)$, which is $CO_{n-2}^+(\R)$. We will use this fact later after exploring the possible types of $V_3\in\Gr_n(\g/\g_3)$.  
   
Below we assume that $\bar x,\bar y\in\langle 
    x_2,\dots,x_{n-1}\rangle$. 
Consider
\[ V_3(X) =  \begin{pmatrix}
 	 	L_1(X)+L_2(X) & L_3(X) & X \\
 	 	(I_{p,q}X)^t & * & 0  \\
 	 	0 & 0 & 0  \\
 	\end{pmatrix}.  \]
Applying Proposition \ref{P:L1}, similarly to the previous cases, we obtain 
 \[ L_2(x_1)=\begin{pmatrix}
   z & 0 & 0  \\
 0 & d & 0 \\
 0 & 0 & -z
 \end{pmatrix}, \,\, L_2(\bar x)=\begin{pmatrix}
      0 & -q^t(\bar x) & 0  \\
    0 & 0 & q(\bar x) \\
    0 & 0 & 0
    \end{pmatrix},   \]	
     where $d\in \gso_{n-2}/\prod_k \gso_{i_k}$, $q_i(\bar x)=0$ if 
     $\alpha_i=0$, $L_2(x_1)=0$ and 
     $L_3=\beta x_n^*\otimes x_1$. Let 
us calculate the restrictions imposed by Proposition \ref{CS} for the current step of Cartan's moving frame method. The Lie bracket of 2 elements in an $n$-dimensional plane $V_3$ is 
\[ [V_3(X),V_3(Y)]= \begin{blockarray}{(ccc)}
T(X,Y) & * & L_2(X)Y- L_2(Y)X  \\
* & * & 0 \\
0 & 0 & 0
\end{blockarray}, \]
where
\[ T(X,Y) = [L_1(X)+L_2(X),L_1(X)+L_2(X)]+L_3(X)Y^t I-L_3(Y)X^t I.\]
Proposition \ref{CS} implies that the symmetric part of $T(X,Y)$ is equal to 
\[L_1( L_2(X)Y- L_2(Y)X).\]
We denote by $Z(X,Y)$ the expression $L_2(X)Y - L_2(Y)X$. Recall that 
 \[L_1(x_1)=0,\,\,
   L_1(x_i)= \alpha_i( x_i^*\ot x_1 + x_n^*\ot x_i),\,\, 
   L_1(x_n)=\sum_i \alpha_i x_i^*\ot x_i. \]
We have to consider 4 cases:
\begin{enumerate}
\item $[V_3(x_1),V_3(x_n)]$. We have $L_1(x_1)=L_2(x_1)=L_3(x_1)=0$ and 
$Z(x_1,x_n)=-L_2(x_n)x_1=-z x_1$. Therefore $T(x_1,x_n)=L_1(-z x_1)=0$. Since
$T(x_1,x_n)=-L_3(x_n)x_1^tI_{n-1,1}=\beta x_1\ot x_n^t$ we conclude $\beta=0$ and hence $L_3=\beta x_n^*\otimes x_1=0$. 
\item $[V_3(\bar x),V_3(\bar y)]$. In this case $L_2(\bar x).\bar y =0$ and $ 
[L_1(\bar x)+L_2(\bar x),L_1(\bar y)+L_2(\bar y)]_{sym}=2\left((\alpha(\bar 
y),q(\bar x))- (\alpha(\bar 
x),q(\bar y))\right)x_n^*\ot x_1=0$, which means that $q$ is a symmetric 
operator with respect to the form $\alpha$, or, in other words, 
$\alpha\circ q$ has a symmetric matrix.
\item $[V_3(x_1),V_3(\bar x)].$ Everything vanishes in this case and 
we do not obtain any restrictions.
\item $[V_3(x_n),V_3(\bar x)]$. In this case $Z(x_n,\bar x)=d\bar x 
- q(\bar x)$ and
\begin{multline*}
[L_1(x_n)+L_2(x_n),L_1(\bar x)+L_2(\bar x)]_{sym}=
[L_1(x_n),L_2(\bar x)]+[L_2(x_n),L_1(\bar x)] =\\
\left[\begin{pmatrix}
     0 & 0 & 0 \\
   0 & \alpha & 0\\
   0 & 0 & 0
   \end{pmatrix}
   ,
  \begin{pmatrix}
    0 & -q^t(\bar x) & 0 \\
    0 & 0 & q^t(\bar x)\\
    0 & 0 & 0
   \end{pmatrix}
\right]
+
\left[\begin{pmatrix}
     z & 0 & 0 \\
   0 & d & 0\\
   0 & 0 & -z
   \end{pmatrix}
   ,
  \begin{pmatrix}
    0 & \alpha^t(\bar x) & 0 \\
    0 & 0 & \alpha(\bar x)\\
    0 & 0 & 0
   \end{pmatrix}
\right]
\\
=  \begin{pmatrix}
    0 & \left(\alpha\circ q(\bar x)+(z+d)\circ \alpha(\bar x)\right)^t 
    & 0 \\
    0 & 0 & \alpha\circ q(\bar x)+(z+d)\circ \alpha(\bar x)\\
    0 & 0 & 0
   \end{pmatrix}.
\end{multline*}
Subtracting $L_2(Z(x_n,\bar x))$ from the expression above we obtain the relation:
\begin{equation}\label{eq5}
2\alpha\circ q + z\alpha + [d,\alpha]=0
\end{equation}
\end{enumerate}

Now we are able to list all surfaces with isotropy of dimension  $\frac{(n-2)(n-3)}2+1$ with $L_1$ of the form $x_1q(\bar x)$.    The dimension of the biggest possible stabilizer for $\alpha$ and therefore of $V_3$ is exactly $\frac{(n-2)(n-3)}2+1$. It can be attained in two cases as was stated above. 

The first case is $CO_{n-3}^+(\R)\ltimes \R^{n-3}$, which stabilizes $\alpha=\diag(1,0,\dots,0)$. In this case 
$q_i(\bar x)=0$, for $i>1$, and 
$d_{ij}=0,$ for $1<i<j$. The fact that $\alpha\circ q$ is symmetric 
implies that $q_{1i}=0$ for $i>0$. Solving  $\eqref{eq5}$ yields
$z=-2 q_{11}$ and $d_{1i}=0$. The stabilizer of $L_2$ defined by these 
conditions coincides with $G_3=CO_{n-2}^+(\R)\ltimes \R^{n+2}$. Therefore Cartan's method stops here. 

It remains to check that the conditions of Proposition \ref{CS} are satisfied. We have
\begin{multline*} T(X,Y)= [L_1( X)+L_2( X),L_1( Y)+L_2(
Y)]_{skew}=
\\
[L_1(X),L_1(Y)]+[L_2(X),L_2(Y)].
\end{multline*}
The only non-trivial brackets in the expression above occur for $X=x_n$, $Y=\bar x$:
\[0=T(x_n,\bar x)-L_2(Z(x_n,\bar x))= \begin{pmatrix}
    0 & -R^t(\bar x) & 0 \\
    0 & 0 & R(\bar x) \\
    0 & 0 & 0
   \end{pmatrix},
 \]
where $R= [d,q]+\alpha^2+q^2+zq$. In this particular case we need to
compute only 
\[R_1=(0+\alpha_1^2+q_{11}^2-2q_{11}^2)\ot 
x_1^*=(\alpha_1^2-q_{11}^2)\ot x_1^*.\]
Since $\alpha_1=1$ we have 2 cases: $q_{11}=1$, $z=-2$ and $q_{11}=-1$, 
$z=2$. The corresponding surfaces are given by:
\begin{equation}\label{srf:4}  
 x_{n+1}=x_1 x_{n} + x_1\sum_{i=2}^{n-1}
x_i^2,
\end{equation}
\begin{equation}\label{srf:5} 
0=x_{n+1} -2x_{n+1}x_n - 2 x_1 x_n^2  + x_1 x_{n} + 
\frac12 \sum_{i=2}^{n-1} x_i^2.
\end{equation}

The second case ($G_3=CO^+_{n-2}(\R)$) requires similar constants $q=1$, $z=-2$ and $q=-1$, 
$z=2$ ($q$ is a scalar operator here) and the resulting surfaces are:
\begin{equation}\label{srf:6} 
 x_{n+1} = x_1 x_{n} + \sum_{i=2}^{n-1} 
x_i^2  + x_1 x_2^2,
\end{equation} 
\begin{equation}\label{srf:7} 
0=
(1-2x_n)(x_{n+1} + x_1 x_{n} +\frac12\sum_{i=2}^{n-1} x_i^2)  + x_n
x_2^2.
\end{equation} 

\section{Generalized type C domains} \label{S:typeC1}

Consider the hypersurface 
$$\Gamma := \{ x \in \R^{n+1} \mid x_{n+1}=x_1x_2 +x_1\sum\limits_{j=3}^n x_j^2, \, \mathrm{and} \, x_1 >0\},$$
which is a trivial reparametrisation of the surface \eqref{typeC1} from Theorem \ref{thm2}. The surface $\Gamma$ is an affinely homogeneous hypersurface that lies on the boundaries of two domains  in $\R^{n+1}$, namely
\begin{align*}
\Omega^{>} &= \{x\in \R^{n+1} \mid x_{n+1}> x_1x_2 +x_1\sum\limits_{j=3}^n x_j^2,\, x_1>0 \} \, \mathrm{and} \\
\Omega^{>} &= \{x\in \R^{n+1} \mid x_{n+1} < x_1x_2 +x_1\sum\limits_{j=3}^n x_j^2,\, x_1>0 \}.
\end{align*}

Let us show that that $\Omega^{>}$ and $\Omega^{<}$ are homogeneous domains in $\R^{n+1}$. 
One can easily verify that $\Omega^{>}$ and $\Omega^{<}$ are invariant under the following group $G$ of affine transformations in $\R^{n+1}$:
\begin{align*}
x_1& \mapsto qx_1 \\
x_2&\mapsto r^2(x_2 - 2 \sum\limits_{j=3}^{n}s_jx_j +t)\\
x_j &\mapsto r(x_j + s_j), \; j=3,\ldots,n \\
x_{n+1} &\mapsto qr^2(x_{n+1}+x_1\sum\limits_{j=3}^{n} s_j^2+ tx_1),
\end{align*}
where  $t, s_j$, for $j=3,\ldots,n$, are arbitrary real parameters,  $q>0$, and $r \neq 0$. 
Applying the transformation above to the point $(1,0,0,\ldots,0,1)\in\Omega^{>}$ gives 
\[ p=\bigl(q, r^2t, rs_3,\ldots, rs_n, qr^2(1+t + \sum\limits_{j=3}^{n} s_j^2)\bigr).\] 
Setting 
\begin{align*}
q &= x_1\\
r &=\frac{\sqrt{x_{n+1}-\bigl(x_1x_2 +x_1\sum\limits_{j=3}^n x_j^2\bigr)}}{\sqrt{x_1}}\\
t &= \frac{x_1x_2}{x_{n+1}-\bigl(x_1x_2 +x_1\sum\limits_{j=3}^n x_j^2\bigr)}\\
s_j &= \frac{x_j\sqrt{x_1}}{\sqrt{x_{n+1}-\bigl(x_1x_2 +x_1\sum\limits_{j=3}^n x_j^2\bigr)}}, \, j=3,\ldots, n,
\end{align*}
we see that $p=(x_1, x_2,x_3,\ldots,x_n)$ can be any prescribed point in $\Omega^{>}$.

Analogously, for $\Omega^{<}$ we take the point $(1,0,0,\ldots,0,-1)$ and setting 
\begin{align*}
q &= x_1\\
r &=\frac{\sqrt{-x_{n+1}+\bigl(x_1x_2 +x_1\sum\limits_{j=3}^n x_j^2\bigr)}}{\sqrt{x_1}}\\
t &= \frac{x_1x_2}{-x_{n+1}+\bigl(x_1x_2 +x_1\sum\limits_{j=3}^n x_j^2\bigr)}\\
s_j &= \frac{x_j\sqrt{x_1}}{\sqrt{-x_{n+1}+\bigl(x_1x_2 +x_1\sum\limits_{j=3}^n x_j^2\bigr)}}, \, j=3,\ldots, n,
\end{align*}
we map it to an arbitrary point $(x_1, x_2,x_3,\ldots,x_n) \in \Omega^{<}$.
It follows that the tube domains 
\begin{align*}
C^{>}& =\{x + \I y \, \mid \,x \in \Omega^{>}, y\in \R^{n+1} \} 
\, \\
C^{<} &=\{x + \I y \, \mid \, x \in \Omega^{<}, y\in \R^{n+1} \}
\end{align*}
in $\C^{n+1}$ are holomorphically homogeneous. The group $G$ considered as a subgroup of affine transformations of $\C^{n+1}$, combined with the translations in the imaginary parts of all variables acts transitively on either $C^{>}$ and $C^{<}$.

We note that neither $C^{>}$ nor $C^{<}$  has a bounded realisation. Indeed, $C^{>}$ contains the complex line $\{z_1=1,z_{n+1}-z_2=1, z_3=\ldots=z_n=0\}$, while $C^{<}$ contains  $\{z_1=1,z_{n+1}-z_2=-1, z_3=\ldots=z_n=0\}$.

Now we will determine the entire holomorphic automorphism groups of $C^{>}$ and $C^{<}$, which we denote by $G^{>}$ and $G^{<}$, respectively. 
Let $\tilde{\Gamma}:= \Gamma + \I \R^{n+1}$. In coordinates $z_j=x_j + \I y_j, \; j=1,\ldots, n+1$ the equation of  $\tilde{\Gamma}$ remains the same as of $\Gamma$ in $\R^{n+1}$, which is    
\begin{equation}\label{tildegamma}
 x_{n+1}=x_1x_2 +x_1\sum\limits_{j=3}^n x_j^2, \,  x_1 >0
\end{equation}
Let $\tilde{G} $ be the group generated by $G$, viewed as a subgroup of affine transformations of $\C^{n+1}$ and by imaginary translations in all coordinates. Clearly, $\tilde{\Gamma}$ is a hypersurface in $\partial C^{>} \cap \partial C^{<}$ with non-degenerate Levi form of the signature $(n-1,1)$. Since the Levi form of $\tilde{\Gamma}$ has both positive and negative eigenvalues every element of $G^{>}$ and $G^{<}$ extends holomorphically to a biholomorphic mapping of a neighbourhood of the reference point $p_0$, preserving $\tilde{\Gamma}$. 

Denote the subgroup of $G$ of the elements with $r=1$ by $H$. This subgroup acts simply transitively on $\Gamma$.
Thus, the subgroup $\tilde{H} \subset \tilde{G}$ that is generated by $H$ and the imaginary translations acts simply transitively on $\tilde{\Gamma}$. Therefore any element  $g$ of either $G^{>}$ and $G^{<}$ can be uniquely represented as $g= h\circ t$, where $h \in  \tilde{H}$ and $t$ is an element of the isotropy subgroups $I_{p_0}^{>} \subset G^{>}$ and $I_{p_0}^{<} \subset G^{<}$ at $p_0$, respectively. 
Therefore, to determine $G^{>}$ and $G^{<}$ it remains to describe $I_{p_0}^{>}$ and $I_{p_0}^{<}$ for an arbitrary point $p_0 \in \tilde{\Gamma}$. 

It is convenient to choose  $p_0=(1,0,0,\ldots,0,0)$ as the base point. Let $I_{p_0}$ be the group of all (local) biholomorphic automorphisms of $\tilde{\Gamma}$, defined in a neighbourhood of $p_0$ and preserving $p_0$. Since, $I_{p_0}^{>}$ and $I_{p_0}^{<}$ consist of global automorphisms, they should belong to $I_{p_0}$. (We will see later that, in fact, $I_{p_0}^{>}=I_{p_0}^{<}=I_{p_0}$). 
If equipped with the topology of uniform convergence on compact subsets of the derivatives of all orders of the component functions, the group $I_{p_0}$ is known to have the structure of a real algebraic group \cite{Bel80}. In the theorem below we determine all generators of $I_{p_0}$. 
It appears more convenient to describe the corresponding Lie algebra $\is_{p_0}$.

\begin{thm}
The Lie algebra $\is_{p_0}$ of the group $I_{p_0}$  is generated by $n^2-4n+7$ holomorphic vector fields $Y_{2n+2}, \ldots, Y_{n^2-2n+8}$ of the form
\begin{align*}
Y_{2n+2}&= 2z_2 \frac{\partial}{\partial z_2}
+ 2z_{n+1} \frac{\partial}{\partial z_{n+1}} 
+ \sum\limits_{j=3}^{n}z_j \frac{\partial}{\partial z_j}\\
R_{j,k}&= z_k \frac{\partial}{\partial z_j}-z_j \frac{\partial}{\partial z_k},\, 3\le j<k \le n,\; \mathrm{forming} \; Y_{2n+3}, \ldots Y_{\frac{n^2-n+8}{2}}\\
I_{j,k}&= \I z_k \frac{\partial}{\partial z_j}+\I z_j \frac{\partial}{\partial z_k},\, 3\le j<k \le n,\; \mathrm{forming} \; Y_{\frac{n^2-n+10}{2}},\ldots, Y_{n^2-3n+8}\\
Y_{n^2-3n+6+j}&= -\I z_j^2 \frac{\partial}{\partial z_2}+\I z_j \frac{\partial}{\partial z_j},\, j=3,\ldots,n,\; \mathrm{forming} \; Y_{n^2-3n+9},\ldots, Y_{n^2-2n+6}\\
Y_{n^2-2n+7}&= 2\I (z_1-1) \frac{\partial}{\partial z_2}
+ \I (z_1^2-1) \frac{\partial}{\partial z_{n+1}} \\
Y_{n^2-2n+8} &=  \frac{\I}{2} (z_1^2-1) \frac{\partial}{\partial z_1} + \I z_{n+1} \frac{\partial}{\partial z_2} + \I z_1z_{n+1} \frac{\partial}{\partial z_{n+1}}.
\end{align*}
\end{thm}

\begin{rem}
The notations $Y_1,\ldots, Y_{2n+1}$ are reserved for the vector fields that act transitively on $\tilde{\Gamma}$, namely:
\begin{align*}
Y_j&=\I \frac{\partial}{\partial z_j}, \, j=1,\ldots,n+1,\;\mathrm{forming} \; Y_{1},\ldots, Y_{n+1}\\
Y_{n+2}&= z_1 \frac{\partial}{\partial z_1}+z_{n+1} \frac{\partial}{\partial z_{n+1}}\\
Y_{n+j}&= -2z_j \frac{\partial}{\partial z_1}+ \frac{\partial}{\partial z_{j}},\; j=3,\ldots,n,\;\mathrm{forming} \; Y_{n+3},\ldots, Y_{2n}\\
Y_{2n+1}&= \frac{\partial}{\partial z_2}+z_{1} \frac{\partial}{\partial z_{n+1}}.
\end{align*}
\end{rem}
\begin{proof}
We first check that the biholomprphic change of variables
\begin{align*}
z_1 &= w_1+1\\
z_2&=w_2 -\frac{\I d_n}{10}w_1w_{n+1}-\sum\limits_{j=3}^{n}\frac{2(w_j)^2}{(2+w_1)^2}\\
z_j&=\frac{2w_j}{2+w_1},\; j=3,\ldots,n\\
z_{n+1}&= -\I w_{n+1}+w_2+\frac{w_1}{2}\bigl(w_2 -\frac{\I d_n}{10}(2+w_1)w_{n+1}\bigr)
\end{align*}
for $d_n=\frac{5(n-2)}{2+n}$ that maps $p$ to the point $0$ in $w$-coordinates 
transforms the equation (\ref{tildegamma}) to the following closed form implicit equation:
\begin{Small}
\begin{equation}\label{chernmosergamma}
\frac{\Im w_{n+1}}{10}=\Re \frac{4(w_1+1)\sum\limits_{j=3}^{n}|w_j|^2 +|w_1|^2(2w_2+w_1\overline{w_2})+4w_1\overline{w_2}+2(w_1)^2\overline{w_2}}{(2+w_1)(2+\overline{w_1})(20-d_nw_1\overline{w_1})}.
\end{equation}
 \end{Small}
We can now show that, in fact, expanding \eqref{chernmosergamma} and resolving it with respect to $\Im w_{n+1}$ in terms of $w'=(w_1,\ldots,w_n)$ yields an equation 
\begin{equation*}
\Im w_{n+1} =
F_{11}(w', \overline{w'}) +\sum\limits_{k,\ell \ge2}F_{k,\ell}(w', 
\overline{w'}), 
\end{equation*} 
where $F_{k,\ell}(w', \overline{w'})$  is a polynomial of bidegree $ 
(k,\ell)$ in Chern-Moser normal form \cite{CM}. 
Indeed, the Levi form of this hypersurface is
\begin{equation*}
F_{11}(w', \overline{w'}) = \frac{1}{2}\bigl(\Re w_1\overline{w_2} + \sum\limits_{j=3}^{n}w_j\overline{w_j}\bigr),
\end{equation*}
and, hence, the Chern-Moser $\tr$ operator is
\begin{equation*}
\tr = 4\bigl( \frac{\partial^2}{\partial w_1 \overline{w_2}} +\frac{\partial^2}{\partial w_2 \overline{w_1}}\bigr)
 + 2\sum\limits_{j=3}^{n}\frac{\partial^2}{\partial w_j \overline{w_j}}
\end{equation*}
The further expansion by jets yields:
\begin{align*}
F_{22}(w', \overline{w'})& =\frac{1}{8}|w_1|^2\bigl(\frac{d_n}{5}\Re w_1\overline{w_2}-\frac{4}{n+2}\sum\limits_{j=3}^{n}|w_j|^2\bigr),\\
F_{32}(w', \overline{w'})& =\frac{1}{16}w_1|w_1|^2\sum\limits_{j=3}^{n}|w_j|^2,\\
F_{33}(w', \overline{w'})& =\frac{d_n|w_1|^4}{160}\bigl(\frac{d_n}{5}(\Re w_1\overline{w_2}+ \sum\limits_{j=3}^{n}|w_j|^2) - \sum\limits_{j=3}^{n}|w_j|^2\bigr).
\end{align*}
The holomorphic vector field $Y_{n^2-2n+8}$ does not contain the variables $z_j, j=3,\ldots,n$ and integrates to the $1$-parametric group $\phi_t$ of symmetries 
\begin{align*}
z_1 & \mapsto \frac{\e^{-\I t}z_1 +\e^{-\I t} + z_1-1} {\e^{-\I t}z_1 +\e^{-\I t} - z_1+1} ,\\
z_2 & \mapsto -\frac{-\e^{-\I t}z_1 z_2 +2\e^{-\I t}z_{n+1} - \e^{-\I t}z_2 +z_2z_1-2z_{n+1}-z_2} {\e^{-\I t}z_1 +\e^{-\I t} - z_1+1},\\
\tilde{z} &   \mapsto\tilde{z}, \; \mbox{for} \; \tilde{z} =(z_3,\ldots,z_n),\\
z_{n+1} & \mapsto \frac{4\e^{-\I t}z_{n+1}} {(\e^{-\I t}z_1 +\e^{-\I t} - z_1+1)^2}.
\end{align*} 
The push-forward 
\begin{small}$$ 
\tilde Y_{n^2-2n+8} = \frac{\I}{2} (w_1^2 + 2w_1)\frac{\partial}{\partial w_1} + (\I w_2 + w_4 + \frac{\I w_1w_2}{2} -\frac{w_1^2w_4}{20} \frac{\partial}{\partial w_{2}} + \frac{\I w_1}{2}\sum\limits_{j=3}{n+1}w_j\frac{\partial}{\partial w_{j}}
$$
\end{small}
of $Y_{n^2-2n+8}$ in normal $w$-coordinates indicates that the
Jacoby matrix $\D \phi_t(0)$ in coordinates $w$ is
$$ 
\D \phi_t(0)= 
\begin{bmatrix}
\e^{\I t} & 0 & 0 &\dots & 0\\
0 & \e^{\I t} & 0 &\dots & 0\\
0 & 0 & 0 & \dots & 0\\
\vdots &\vdots &\ddots & 0\\
 0 & 0 & 0 & \dots & 0
\end{bmatrix}.
$$

In these coordinates the isotropy group $I_0$ (that corresponds to $I_p$ at the original base point) contains the subgroup $I$, consisting of $\phi_t(w)$ and
\begin{align*}
w_1 & \mapsto w_1,\\
w_2 & \mapsto r^2(w_2 + \I u w_1),\\
\tilde{w} & \mapsto r U  \tilde{w}, \; \mbox{for} \; \tilde{w} =(w_3,\ldots,w_n),\\
w_{n+1} & \mapsto r^2 w_{n+1},
\end{align*}
for $r \neq 0, \, u \in \R$ and  some unitary {\small $(n-2)\times (n-2)$} matrix $U$ that preserves the positive definite hermitian form $\sum\limits_{j=3}^{n}|w_j|^2$. 
We show, that, in fact, $I_0= I$. 

Let $\Phi \in I_0$. Then the pseudo-unitary linear transformation $\D \Phi(0)$ has to preserve $F_{22}(w', \overline{w'})$. 
The Lie algebra of linear holomorphic vector fields that preserve $F_{22}(w', \overline{w'})$ consists of all fields of the form $\mathcal X = w X \frac{\partial}{\partial w}$, such that $\Re \mathcal X F_{22}  = \lambda F_{22}, \quad \lambda \in \R$. 
The general form of the matrix $X$ is 
$$
X = \begin{bmatrix}
r_{11} + \I s_{11} & \I s_{12} & -2r_{32} + 2\I s_{32} &\dots &\dots & -2r_{n2} + \I s_{n2}\\
\I s_{21} & -r_{11} + \I s_{11} & 2r_{31} + 2\I s_{31} &\dots & \dots &-2r_{n1} + \I s_{n1}\\
r_{31} + \I s_{31} &  r_{32} + \I s_{32} &\I s_{33}  & x_{34} & \dots & x_{3n} \\
r_{41} + \I s_{41} &  r_{42} + \I s_{42} & -\overline{x_{34}} & \I s_{44}  & \dots & x_{3n} \\
\vdots &\vdots &\vdots & \vdots &\vdots & \vdots \\
r_{n1} + \I s_{n1} &  r_{n2} + \I s_{n2} & -\overline{x_{3n}} & -\overline{x_{4n}}  & \dots & \I s_{nn}
\end{bmatrix}.
$$
Subtracting the fields of the known isotropy $I$ yields
$$
X = \begin{bmatrix}
0 & \I s_{12} & -2r_{32} + 2\I s_{32} &\dots &\dots & -2r_{n2} + \I s_{n2}\\
0 & 0 & 2r_{31} + 2\I s_{31} &\dots & \dots &-2r_{n1} + \I s_{n1}\\
r_{31} + \I s_{31} &  r_{32} + \I s_{32} &0  & 0 & \dots & 0 \\
r_{41} + \I s_{41} &  r_{42} + \I s_{42} & 0  & 0 & \dots & 0 \\
\vdots &\vdots &\vdots & \vdots &\vdots & \vdots \\
r_{n1} + \I s_{n1} &  r_{n2} + \I s_{n2} &0  & 0 & \dots & 0
\end{bmatrix}.
$$
Considering the form of $F_{22}$, the condition $\Re \mathcal X F_{22}  =  w X \frac{\partial}{\partial w} F_{22} = 0$ implies that the first row of $X$ must be zero. Therefore, its second row is zero as well, and, hence, $X=0$. Thus, the Jacobi matrices of all known isotropy automorphisms form the entire subgroup of unitary transformations that preserve $F_{22}$. Thus, $I_0 \subset I$, and, therefore $I_0=I$.

The isotropy transformations in $I_p$ that correspond to elements of $I_0$ of the form
\begin{align*}
w_1 & \mapsto w_1,\\
w_2 & \mapsto r^2w_2 ,\\
\tilde{w} & \mapsto r   \tilde{w}, \; \mbox{for} \; \tilde{w} =(w_3,\ldots,w_n),\\
w_{n+1} & \mapsto r^2 w_{n+1},
\end{align*}
are obviously the elements of the affine subgroup $G$ of the same form
\begin{align*}
z_1 & \mapsto z_1,\\
z_2 & \mapsto r^2z_2 ,\\
\tilde{z} & \mapsto r   \tilde{z}, \; \mbox{for} \; \tilde{z} =(z_3,\ldots,z_n),\\
z_{n+1} & \mapsto r^2 z_{n+1}.
\end{align*}
The isotropy in $I_p$ that corresponds to the generator 
\begin{align*}
w_1 & \mapsto w_1,\\
w_2 & \mapsto w_2 + \I u w_1,\\
\tilde{w} & \mapsto  \tilde{w}, \\
w_{n+1} & \mapsto  w_{n+1},
\end{align*}
of $I_0$ is 
\begin{align*}
z_1 & \mapsto z_1,\\
z_2 & \mapsto z_2 + \I u (z_1-1),\\
\tilde{z} & \mapsto  \tilde{z}, \\
z_{n+1} & \mapsto  z_{n+1} +\frac{\I u((z_1)^2-1)}{2}.
\end{align*}
The elements of $I_p$ that correspond to the generators
\begin{align*}
w_1 & \mapsto w_1,\\
w_2 & \mapsto w_2 ,\\
w_{3+j} & \mapsto \e^{I \theta} w_{3+j}, \mbox{for some}\;  j \in \{0,1,\ldots,n-3\} \\
w_{3+k} & \mapsto  w_{3+k}, \mbox{for} \; k=0,\ldots,n-2, \; k\neq j 
\end{align*}
reduce to 
\begin{align*}
z_1 & \mapsto z_1,\\
z_2 & \mapsto z_2 - \frac{(\e^{2\I \theta} -1)(z_{3+j})^2}{2} ,\\
z_{3+j} & \mapsto \e^{\I \theta} z_{3+j}, \mbox{for some} \; j \in \{0,1,\ldots,n-3\} \\
z_{3+k} & \mapsto  z_{3+k}, \mbox{for} \; k=0,\ldots,n-2, \; k\neq j.
\end{align*}
The generators of $I_0$ of the form 
\begin{align*}
w_1 & \mapsto w_1,\\
w_2 & \mapsto w_2 ,\\
w_{j} & \mapsto \cos(\theta) w_{j} + \sin(\theta) w_{k} \; \mbox{for some pair}\;  j,k \in \{3,\ldots,n\}, j<k \\
w_{k} & \mapsto -\sin(\theta) w_{j} + \cos(\theta) w_{k} \\
w_{\ell} &\mapsto w_{\ell}, \; \mbox{for all other}\; \ell \neq j \; \mbox{or} \; k 
\end{align*}
correspond to rotations 
\begin{align*}
z_1 & \mapsto z_1,\\
z_2 & \mapsto z_2 ,\\
z_{j} & \mapsto \cos(\theta) z_{j} + \sin(\theta) z_{k} \; \mbox{for some pair}\;  j,k \in \{3,\ldots,n\}, j<k \\
z_{k} & \mapsto -\sin(\theta) z_{j} + \cos(\theta) z_{k} \\
z_{\ell} &\mapsto z_{\ell}, \; \mbox{for all other}\; \ell \neq j \; \mbox{or} \; k. 
\end{align*} 
The generators of $I_0$ of the form 
\begin{align*}
w_1 & \mapsto w_1,\\
w_2 & \mapsto w_2 ,\\
w_{j} & \mapsto \cos(\theta) w_{j} + \sin(\theta) w_{k} \; \mbox{for some pair}\;  j,k \in \{3,\ldots,n\}, j<k \\
w_{k} & \mapsto -\sin(\theta) w_{j} + \cos(\theta) w_{k} \\
w_{\ell} &\mapsto w_{\ell}, \; \mbox{for all other}\; \ell \neq j \; \mbox{or} \; k 
\end{align*}
correspond to  
\begin{align*}
z_1 & \mapsto z_1,\\
z_2 & \mapsto z_2 ,\\
z_{j} & \mapsto \cos(\theta) z_{j} + \I \sin(\theta) z_{k} \; \mbox{for some pair}\;  j,k \in \{3,\ldots,n\}, j<k \\
z_{k} & \mapsto \I \sin(\theta) z_{j} + \cos(\theta) z_{k} \\
z_{\ell} &\mapsto z_{\ell}, \; \mbox{for all other}\; \ell \neq j \; \mbox{or} \; k. 
\end{align*}
And, finally, 
the holomorphic vector field $Y_{n^2-2n+8}$ integrates to the $1$-parameter subgroup 
\begin{align*}
z_1 & \mapsto \frac{\e^{-\I t}z_1 +\e^{-\I t} + z_1-1} {\e^{-\I t}z_1 +\e^{-\I t} - z_1+1} ,\\
z_2 & \mapsto -\frac{-\e^{-\I t}z_1 z_2 +2\e^{-\I t}z_{n+1} - \e^{-\I t}z_2 +z_2z_1-2z_{n+1}-z_2} {\e^{-\I t}z_1 +\e^{-\I t} - z_1+1},\\
\tilde{z} &   \mapsto\tilde{z}, \; \mbox{for} \; \tilde{z} =(z_3,\ldots,z_n),\\
z_{n+1} & \mapsto \frac{4\e^{-\I t}z_{n+1}} {(\e^{-\I t}z_1 +\e^{-\I t} - z_1+1)^2}.
\end{align*} 
As we see, all listed generators of $I_p$ are polynomial and birational transformations that give rise to the vector fields 
$Y_1, \ldots, Y_{n^2-2n +8}$ that generate $\is_{p_0}$.
\end{proof}

Our final result shows that the domains $C^{>}$ and $C^{<}$ are not Penney nil-balls.
\begin{thm}
Neither of two domains $C^{>}$ or $C^{<}$ is equivalent to a nil-ball.
\end{thm}
\begin{proof}
If either of these domains was equivalent to a nil-ball then $\tilde \Gamma$ would admit a transitive action of a nilpotent group (see \cite{P1}, \cite{P3}). This implies that the Lie algebra $\g$ of $G^{>} = G^{<}$ would have a nilpotent subalgebra that acts transitively on $\tilde \Gamma$. We will show that such a subalgebra does not exist.

The following three vector fields form an $\gsl_2(\R)$ subalgebra of $\g$:
\begin{align*}
&\frac\I4 \bigl( (z_1^2-1)\frac{\partial}{\partial z_{1}} + 2z_{n+1} \frac{\partial}{\partial z_{n}}+2z_{n+1} z_1 \frac{\partial}{\partial z_{n+1}}\bigr)= \frac12 Y_{n^2-2n+8}
\\
&z_1\frac{\partial}{\partial z_{1}}+z_{n+1}\frac{\partial}{\partial z_{n+1}} = Y_{n+2}
\\
&\frac\I4 \bigl( z_1^2\frac{\partial}{\partial z_{1}} + 2z_{n+1}\frac{\partial}{\partial z_{n}}+2z_{n+1} z_1\frac{\partial}{\partial z_{n+1}}\bigr) = \frac12 Y_{n^2-2n +8} + \frac14 Y_1.
\end{align*}
The isotropy contains only the first
element. Therefore the projection of a transitive nilpotent subalgebra of $\g$ on this $\gsl_2(\R)$, if such exists, would be 2-dimensional and nilpotent. This is impossible.
\end{proof}

\end{document}